\documentclass[10PT]{amsart}

\usepackage{latexsym}
\usepackage{amssymb}
\usepackage{amsmath}
\usepackage{amsfonts}
\usepackage[mathscr]{eucal}

\usepackage{amscd}

\usepackage{graphicx}
\usepackage{epsfig}

\usepackage{amsthm}
\usepackage[all,cmtip]{xy}
\usepackage{hyperref}
\usepackage{pst-grad} 
\usepackage{pst-plot} 
\usepackage{comment}

\newtheorem{thm}{Theorem}[section]

\newtheorem{lem}[thm]{Lemma}
\newtheorem{prop}[thm]{Proposition}

\newtheorem{doublesoulthm}[thm]{Double Soul Theorem}

\newtheorem*{main_thm}{Main Theorem}

\newtheorem*{thm*}{Theorem}

\theoremstyle{definition}

\theoremstyle{remark}

\newtheorem*{ack}{Acknowledgments}

\numberwithin{equation}{section}

\newcommand{\curv}{\mathrm{curv}}

\newcommand{\sphere}{\mathrm{\mathbb{S}}}

\newcommand{\Int}{\mathrm{\mathbb{Z}}}
\newcommand{\SO}{\mathrm{SO}}
\newcommand{\SU}{\mathrm{SU}}

\newcommand{\Hh}{\mathrm{H}}
\newcommand{\G}{\mathrm{G}}

\newcommand{\Id}{\mathrm{1}}
\newcommand{\Ss}{\mathrm{S}}

\newcommand{\Fix}{\mathrm{Fix}}

\begin{document}



\title[Nonnegatively curved fixed point homogeneous $5$-manifolds]{Nonnegatively curved fixed point homogeneous $5$-manifolds}

\author[F.\ Galaz-Garcia and W.\ Spindeler]{Fernando Galaz-Garcia and Wolfgang Spindeler$^*$}

\thanks{$^*$ Supported by SFB 878 - Groups, Geometry \& Actions}

\address{Mathematisches Institut, WWU M\"unster, Germany}
\email{f.galaz-garcia@uni-muenster.de}
\email{wolfgang.spindeler@uni-muenster.de}

\date{\today}




\begin{abstract}
Let $\G$ be a compact Lie group acting effectively by isometries on a compact Riemannian manifold $M$ with nonempty fixed point set $\Fix(M,\G)$. We say that the action is \emph{fixed point homogeneous} if $\G$ acts transitively on a normal sphere to some component of $\Fix(M,\G)$, equivalently, if $\Fix(M,\G)$ has codimension one in the orbit space of the action. We classify up to diffeomorphism closed, simply connected $5$-manifolds with nonnegative sectional curvature and an effective fixed point homogeneous isometric action of a compact Lie group. 
\end{abstract}

\maketitle



%
%

\section{Introduction and main result}


The classification of closed Riemannian manifolds with positive or nonnegative (sectional) curvature is a fundamental open problem in Riemannian geometry. In this context, the classification of these manifolds in the presence of a non-trivial symmetry group can be regarded as a first step towards more general classification results. One is led in this way to consider positively or nonnegatively curved closed Riemannian manifolds with an (effective) isometric action of a compact Lie group. ``Large'' actions, interpreted in different ways (cf. \cite{Gr2002,Wkg2007}), have received particular attention in the literature, e.g., \cite{Wa}, \cite{BB}, \cite{Wi_sym}, \cite{S},  \cite{Ve2002}, \cite{Ve2004}, \cite{GZ1},  \cite{GWZ}, \cite{GroveSearle1997}. 

Let $M$ be a smooth manifold with an (effective) smooth action of a compact Lie group $\G$. One possible measure for the size of the action $\G\times M \rightarrow M$ is its  \emph{cohomogeneity}, defined as the dimension of the orbit space $M/\G$. Under this interpretation, the largest actions will be those for which $\dim M/\G=0$, i.e., the action is transitive and $M$ is a homogeneous space. If the action has fixed points, $\dim M/\G$ is bounded below by the dimension of the fixed point set $\Fix(M,\G)$ and
\[
	\dim M/\G\geq \dim\Fix(M,\G) +1
\]
for any non-trivial action. In this case the {\it fixed point cohomogeneity} of the action, denoted by $\operatorname{cohomfix}(M,\G)$, is defined by
\[
\operatorname{cohomfix}(M,\G) = \dim M/\G - \dim\Fix(M,\G) -1\geq 0.
\]
For an action with fixed points, ``large'' may be interpreted as having low fixed point cohomogeneity. If the fixed point cohomogeneity of the action is $0$, we say that the action is \emph{fixed point homogeneous} and we say that $M$ is a {\it fixed point homogeneous manifold} (cf. \cite{GroveSearle1997}). Observe that the fixed point set of a fixed point homogeneous action has codimension $1$ in the orbit space.

 Grove and Searle \cite{GroveSearle1997} classified closed (Riemannian) manifolds with positive curvature and a fixed point homogeneous isometric Lie group action, along with the possible actions. In the simply connected case one has the following topological classification:

\begin{thm*}[Grove, Searle]\label{thm_pos_fph} Any closed, simply connected fixed point homogeneous manifold with positive curvature is diffeomorphic to a compact rank one symmetric space. 
\end{thm*}

This result has been used, for example, in the classification of simply connected, positively curved closed manifolds of cohomogeneity one (cf. \cite{S,GWZ,Ve2002,Ve2004}), where fixed point homogeneous actions arise naturally. 

The class of fixed point homogeneous manifolds with nonnegative curvature properly contains the collection of compact rank one symmetric spaces. Indeed, given a  nonnegatively curved manifold $M$ and a positively curved compact rank one symmetric space $N$ with an isometric fixed point homogeneous $\G$-action, one may construct a nonnegatively curved fixed point homogeneous $\G$-manifold by taking the product $M\times N$ equipped with the product metric and letting $\G$ act trivially on $M$ and fixed point homogeneously  on $N$. It follows, in particular, that any nonnegatively curved manifold may occur as a fixed point set component of a fixed point homogeneous action.

 As the paragraph above illustrates,  the general classification of closed fixed point homogeneous  manifolds with nonnegative curvature is a difficult problem. Nevertheless, in dimensions at most $5$, where fixed point set components have dimension at most $3$, a complete topological classification can be given in the simply connected case. In dimensions $4$ and below the topological classification was carried out in  \cite{Ga2}. In this paper we address the topological classification in dimension $5$. Our main result is the following theorem. 


\begin{main_thm}
\label{thm:main} Let $M^5$ be a closed, simply connected $5$-dimensional nonnegatively curved fixed point homogeneous $\G$-manifold. Then $\G$ is $\SO(5)$, $\SO(4)$, $\SU(2)$, $\SO(3)$ or $\Ss^1$ and we have the following classification. 
\\

\begin{itemize}
\item[\emph{(a)}] If $\G=\SO(5)$, $\SO(4)$ or $\SU(2)$, then $M^5$ is diffeomorphic to $\sphere^5$.
\bigskip
\item[\emph{(b)}] If $\G=\SO(3)$ or $\Ss^1$, then $M^5$ is diffeomorphic to $\sphere^5$ or to one of the two bundles over $\sphere^2$ with fiber $\sphere^3$.
\end{itemize}
 
\end{main_thm}

Observe that the list of fixed point homogeneous $5$-manifolds in the Main Theorem contains every known closed, simply connected $5$-manifold of nonnegative sectional curvature except for the Wu manifold $\SU(3)/\SO(3)$.


We prove the Main Theorem in Section~\ref{S:Proof_Main_Thm}, after recalling  some basic facts about group actions and Alexandrov spaces in Section~\ref{S:Prelim}. To prove the theorem we proceed as follows. When $\G$ is one of $\SO(5)$, $\SO(4)$, $\SU(2)$ or $\SO(3)$ the result follows easily from classification results in the literature. When $\G$ is $\Ss^1$ there are no general topological classification results for closed simply connected $5$-manifolds with a smooth circle action (see \cite{Ko2006}, though). In our case, the hypothesis of nonnegative curvature allows us to show, by looking at the orbit space structure, that $M^5$ decomposes as a union of two disc bundles over smooth submanifolds of $M^5$, one of which is a $3$-dimensional component of $\Fix(M^5,\Ss^1)$.  This in turn allows us to show, after some work, that  $H_2(M^5, \mathbb{Z})$ is either $0$ or $\mathbb{Z}$, whence the conclusion follows from the Barden-Smale classification of smooth, closed simply connected $5$-manifolds \cite{Ba1965,Sm1962}.



\begin{ack}The authors would like to thank Burkhard Wilking for helpful conversations. The second named author would like to thank Karsten Grove for numerous conversations on fixed point homogeneous manifolds.
\end{ack}

%
%

\section{Preliminaries}
\label{S:Prelim}

In this section we will recall some basic facts and fix  notation. We will consider all manifolds to be smooth and closed, i.e., compact and without boundary, and all actions will be assumed to be smooth and effective, unless stated otherwise. We will write $M\cong N$ to denote that two manifolds, $M$ and $N$, are diffeomorphic. Similarly, when two groups, $G$ and $H$, are isomorphic, we will write $G\cong H$.


Let $\G$ be a Lie group acting (on the left) on a smooth manifold $M$.  We denote by $\G_x$ the isotropy group at $x\in M$ and by $\G x$ the orbit of $x$. 
We will denote the fixed point set of the action by  
$\Fix(M , \G )$ and define its dimension as 
$$
	\dim(\Fix(M,\G))=\max \{\,\dim(N): \text{$N$ is a component of $\Fix(M,\G)$}\,\}.
$$
Given a subset $A\subset M$, we will denote its image in $M/\G$ under the orbit projection map $\pi:M\rightarrow M/\G$ by $A^*$. Following this notation, we will denote the orbit space of the action $\G\times M\rightarrow M$ by $M^*$, and a point in $M^*$ will be denoted by $p^*$, coressponding to the orbit $\G p$. We recall the well-known fact that the orbit space of a Lie group action on a simply connected manifold is simply connected (cf. \cite{Br1972}, Ch. 4).


Recall that a finite dimensional length space $(X,\mathrm{dist})$ is an \emph{Alexandrov space} if it has curvature bounded from below (cf. \cite{BBI}).
 When $M$ is a complete, connected Riemannian manifold and $\G$ is a compact Lie group acting on $M$ by isometries, the orbit space $M^*$ is equipped with the orbital distance metric induced from $M$, i.e., the distance between $p^*$ and $q^*$ in $M^*$ is the distance between the orbits $\G p$ and $\G q$ as subsets of $M$.  If, in addition, $M$ has sectional curvature bounded below, that is, $\sec M\geq k$, then the orbit space $M^*$ is an Alexandrov space with $\curv M^* \geq k$. 
 
The \emph{space of directions} of a general Alexandrov space at a point $x$ is
by definition the completion of the 
space of geodesic directions at $x$. In the case of orbit spaces $M^*=M/\G$, the space of directions $\Sigma_{p^*}M^*$ at a point $p^*\in M^*$ consists of geodesic directions and is isometric to
\[
S^{\perp}_p/\G_p,
\] 
where $S^{\perp}_p$ is the unit normal sphere to the orbit $\G p$ at $p\in M$.

When $M$ is a nonnegatively curved, fixed point homogeneous, Riemannian $\G$-manifold, the orbit space $M^*$ is a nonnegatively curved Alexandrov space and $\partial M^*$ contains a component $F$ of $\Fix(M,\G)$. Let $C^*\subset M^*$ denote the set at maximal distance from $F^*\subset \partial M^*$ and let $C=\pi^{-1}(C^*)$. It follows from the Soul Theorem for orbit spaces (cf. \cite{Gr2002, GroveSearle1997})  that $M$ can be written as the union of neighborhoods $D(F)$ and $D(C)$ along their common boundary $E$, i.e., 
\[
M=D(F)\cup_E D(C).
\]
In particular, when $C$ is another fixed point set component with maximal dimension, the Double Soul theorem in  \cite{SY} easily generalizes to the following.


\begin{doublesoulthm}[Searle, Yang]
\label{thm:double_soul_thm}
 Let $M$ be a nonnegatively curved fixed-point homogeneous Riemannian $\G$-manifold with principal isotropy $\Hh$. If $\mathrm{Fix}(M,\G)$ contains at least two components $F,N$ with maximal dimension, one of which is
compact, then $F$ and $N$ are isometric and  $M$ is diffeomorphic to an $\sphere^{k+1}$-bundle over $F$, where $\sphere^k=\G/\Hh$.

\end{doublesoulthm}

%
%

\section{Proof of the Main Theorem}
\label{S:Proof_Main_Thm}

 We have divided the proof  into four parts, corresponding to the cohomogeneity of the action. 
 
 Let $M^5$ be a simply connected nonnegatively curved fixed point homogeneous $5$-manifold. Let $\sphere^k$ be a normal sphere to a component of $\Fix(M^5,\G)$ with maximal dimension. We have  $1\leq k\leq 4$ and the cohomogeneity of the action is $5-k$. It suffices to consider (cf. \cite{GroveSearle1997}) the following pairs $(\G, \Hh)$ of Lie groups $\G$ acting transitively with principal isotropy $\Hh$ on $\sphere^k$, $1\leq k\leq 4$. 
\\

\begin{equation*} 
\label{eq:fx_pt_groups}
(\G, \Hh) = \left\{ 
\begin{array}{ll} 
(\SO(5),\SO(4)), & \text{if } k =4;\\[.2cm]
(\SU(2),\SU(1))\text{ or }(\SO(4),\SO(3)), & \text{if } k = 3;\\[.2cm] 
(\SO(3),\SO(2)), & \text{if } k=2;\\[.2cm]
(\Ss^1,\mathsf{1}), & \text{if } k=1.
\end{array} \right. 
\end{equation*}
\bigskip


\subsection{Cohomogeneity $1$}



By work of Hoelscher \cite{Ho2007}, a simply connected $5$-manifold with a cohomogeneity one action must be diffeomorphic to $\sphere^5$, the Wu manifold $\SU(3)/\SO(3)$ or to one of the two $\sphere^3$-bundles over $\sphere^2$. Each manifold in this list admits a metric with nonnegative curvature invariant under the cohomogeneity one action. In the fixed point homogeneous case it was shown in \cite{GroveSearle1997} that $M^5$ is equivariantly diffeomorphic to $\sphere^5$ equipped with a linear action. We point out that, more generally and without any curvature assumptions, a closed, simply connected manifold with a fixed point homogeneous action of cohomogeneity one must be equivariantly diffeomorphic to a compact rank one symmetric space equipped with one of the isometric actions described in \cite{Ho2007} (see also \cite{GroveSearle1997}). 





\subsection{Cohomogeneity $2$}


We will show, more generally, that a simply connected, fixed point homogeneous $n$-manifold $M^n$ of cohomogeneity two is diffeomorphic to $\sphere^n$. This will be a consequence of the following result  (cf. Bredon \cite{Br1972} Chapter 4). Here we denote the set of singular orbits by $B$, the set of exceptional orbits by $E$, and  the set of special exceptional orbits by $SE$. 


\begin{lem}
\label{L:cohom2_Bredon} Let $M$ be a smooth compact manifold with a smooth action of a compact connected Lie group. 
\begin{itemize}
	\item[\emph{(a)}] If $H_1(M, \Int_2)=0$ and a principal orbit is connected, then $SE=\emptyset$.\\
	\item[\emph{(b)}] If the action has cohomogeneity two, then the orbit space $M^*$ is a $2$-manifold and $\partial M^*=B^*\cup \overline{SE}^*$.\\
	\item[\emph{(c)}] If the action has cohomogeneity two, $H_1(M, \Int)=0$ and there are singular orbits, then $E^*=\emptyset$ and $M^*$ is a $2$-disc with boundary $B^*$.
\end{itemize}
\end{lem}
\bigskip

It follows from Lemma~\ref{L:cohom2_Bredon} that the orbit space $M^*$ is a $2$-disc whose boundary consists of fixed points and its interior consists of principal orbits. It then follows from \cite{GroveSearle1997} that $M^n$ is diffeomorphic to $\sphere^n$. Thus we have the following proposition. 

\begin{prop}[Cohomogeneity two classification] A closed, simply connected $n$-manifold with a fixed point homogeneous action of cohomogeneity two is diffeomorphic to  $\sphere^n$. 
\end{prop}


\subsection{Cohomogeneity $3$}


Since $M^5$ is simply connected, it follows from Lemma~\ref{L:cohom2_Bredon} that the only orbit types are fixed points and principal orbits. In this case the orbit space $M^*$ is  a simply connected $3$-manifold with a boundary component $F^2$, corresponding to a fixed point set component with maximal dimension. Suppose first that there are no other fixed point set components. Observe that, by Perelman's resolution of the Poincar\'e Conjecture \cite{Pe2002,Pe2003}, $F^2$ is diffeomorphic to $\sphere^2$. It follows from \cite{GroveSearle1997} that $M^5$ is diffeomorphic to $\sphere^n$. If there is a second $3$-dimensional fixed point set component, it follows from the Double Soul Theorem~\ref{thm:double_soul_thm} that  $M^5$ is diffeomorphic to one of the two $\sphere^3$ bundles over $\sphere^2$.


\subsection{Cohomogeneity $4$}


In this case $(\G,\Hh)=(\Ss^1, \Id)$. The orbit space $M^*$ is a $4$-dimensional, simply connected nonnegatively curved Alexandrov space. Fixed point set components are $3$- or $1$-dimensional totally geodesic closed submanifolds of $M^5$ and, by Lemma~\ref{L:cohom2_Bredon}, there are no special exceptional orbits. There may be exceptional orbits with finite isotropy $\Int_k$ for some $k\geq 2$, and each connected component $A$ of $\Fix(M,\Int_k)$ is a totally geodesic submanifold of $M^5$ of even codimension. Hence the projection $A^*$ of $A$ in the orbit space $M^*$ is either $2$-dimensional or an isolated point. We will say that a neighborhood $U^*$ of an isolated point $p^*\in M^*$ with non-trivial isotropy is a \emph{regular neighborhood} if $U\setminus\{\,p\,\}$ consists only of principal orbits. 
\\

Let $F^*$ be a component of maximal dimension of $\Fix(M,\Ss^1)$, and hence a component of $\partial M^*$.  We let $C^*$ be the set at maximal distance from $F^*$ in the orbit space $M^*$. We have $0\leq \dim C^*\leq 3$. Recall that, by the Soul Theorem for orbit spaces, $\operatorname{dist}(F^*,\cdot)$ is concave and $M^*$ deformation retracts onto $C^*$. Moreover, all the points in $M^*\setminus \{F^*\cup C^*\}$ correspond to principal orbits (cf. \cite{GroveSearle1994,GroveSearle1997}). 
\\


The following result from \cite{GaGaSe2010} will simplify our analysis when $C^*$ has dimension $1$ or $3$.


\begin{lem}
\label{L:FPH_SC} Let $M^n$ be a closed, simply connected, nonnegatively curved manifold of dimension $n\geq 4$ with an isometric $\Ss^1$-action and suppose that $\Fix(M^n, \Ss^1)$ contains an $(n-2)$-dimensional component $F$.  Let $C^*$ be the set at maximal distance from $F^*$ in the orbit space $M^*$.
\begin{itemize}
	\item[\emph{(a)}] If $\dim C^*=n-2$,  then $C$ is fixed by the $\Ss^1$ action, $C^* = C$ is isometric to $F^* = F$ and is simply connected.\\
	\item[\emph{(b)}] If $\dim C^*\leq n-4$, then $F^*$ is simply connected.\\
\end{itemize}
\end{lem}

The following proposition yields further information on the topology of $C^*$ in our particular case. 


\begin{prop}
\label{P:restr_C}Let $M$ be a closed, simply connected nonnegatively curved $5$-manifold with a fixed point homogenous $\Ss^1$-action.
\begin{itemize}
\item[\emph{(a)}] If $\Fix(M,\Ss^1)$ contains a $1$-dimensional component $N$, then $\dim C^* = 2$, $N^* = \partial C^*$ and $C^*$ is homeomorphic to  a $2$-disc.\\
\item[\emph{(b)}] If $C^*$ contains an isolated exceptional orbit $p^*$, then the preimage under the orbit projection map of an open regular neighborhood of $p^*$ in $C^*$ is a submanifold of $M$.\\
\item[\emph{(c)}] If $\partial C^*$ contains an isolated exceptional orbit $p^*$, then $\mathrm{S}^1_p = \mathbb{Z}_2$ and $C^*$ is homeomorphic to a closed interval. 
\end{itemize}
\end{prop}

\begin{proof} Let $p^* \in N^*\cong \sphere^1$. Consider the tangent cone $K_{p^*}C^*$ of $C^*$ at $p^*$. Since $K_{p^*}N^*\subset K_{p^*}C^*$ is isometric to $\mathbb{R}$,   
it follows from the Splitting Theorem for Alexandrov spaces that $K_{p^*}C^*$ is isometric to a direct product $\mathbb{R} \oplus W$, where $W$ is contained in $(\nu_p N)/\Ss^1$ and $\nu_pN$ denotes the normal bundle of $N$ at $p$. There exists some orbit $\Ss^1u \subset \sphere^3 \subset \nu_p N$ wich has distance $\geq \pi /2$ to all orbits in $\sphere^3$ whose corresponding directions are contained in $W$. Since $\Ss^1$ is acting without fixed points on $\sphere^3$ it follows that $\dim W \leq 1$ and thus $\dim C^* \leq 2$. Because $N^* \cong \sphere^1$ is contained in $C^*$ we cannot have $\dim C^* = 1$, since $M^*$ is simply connected. Further, no shortest path can cross $N^*$ by Kleiner's isotropy lemma. Since $N^*$ has codimension $1$ we see that $N^* \subseteq \partial C^*$ and equality holds again since $M^*$ is simply connected. This proves (a).

We now prove parts (b) and (c). Let $\mathbb{Z}_k$ be the isotropy group of $p^{*}$ and let $(S,g)$ be a slice at $p$ with a metric $g$ such that $(S,g)/\mathbb{Z}_k$ is isometric to $B_r(p^{*})$, the open ball of radius $r$ around $p^*$. Consider the map 
$$
h:(S,g) \rightarrow \mathbb{R},\ q \mapsto d(\partial F^{*},q^*).
$$
Let $\tilde{C} := \pi^{-1}(C^*) \cap S$ and observe that $\tilde{C}= \{\, q\in S : h(q)\ \text{maximal} \,\}$.
Note that $h$ is concave, so $\tilde{C}$ is a locally convex subset of $S$ and hence a submanifold of $S$ with smooth interior and possibly nonsmooth boundary by the work of Cheeger and Gromoll (see \cite{ChGr1972}). Assume now that $p \in \partial \tilde{C}$. Since $\tilde{C}$ is invariant under the action of $\mathbb{Z}_k$ on $S$, the space of directions $\Sigma_{p}\tilde{C}$ is invariant under the free action of $\mathbb{Z}_k$ on $\sphere^3$, the unit normal sphere to $\Ss^1p$. Since $\Sigma_{p}\tilde{C}$ is a positively curved Alexandrov space, there is a unique point $v \in \Sigma_{p}\tilde{C}$ at maximal distance from $\partial \Sigma_{p}\tilde{C}$, and $v$ must be fixed by $\mathbb{Z}_k$. This is a contradiction to $\mathbb{Z}_k$ acting freely. Therefore $p \in \operatorname{Int}(\tilde{C})$. After choosing a smaller $r$, if necessary, we have $\partial \tilde{C} = \emptyset$. We then have that $\pi^{-1}(B_r(p^{*})) = \Ss^1 (\tilde{C})$ is a smooth submanifold of $M$ with empty boundary, which proves (b). 

To prove (c), observe that $\Sigma_{p^{*}}C^* = \Sigma_{p}\tilde{C}/\mathbb{Z}_k = \sphere^n/\mathbb{Z}_k,$ with $n = \dim C^* -1$ and $\mathbb{Z}_k$ acting freely. Hence $\Sigma_{p^{*}}C^*$ has boundary if and only if $\dim C^* = 1$ and $k = 2$.
\end{proof}

We now show that $M^5$ decomposes as the union of two disc bundles.


\begin{prop}
Let $M$ be a closed, simply connected nonnegatively curved $5$-manifold with a fixed point homogenous $\Ss^1$-action and let $F$ be a three-dimensional component of $\operatorname{Fix}(M,\Ss^1)$.  Then there exists a closed, invariant and orientable submanifold $H \subset M$ with $\dim H \in \{1,3\}$ such that $M$ is diffeomorphic to the union of two distance tubular neighbourhoods of $F$ and $H$ glued together along their common boundary $E$, i.e.,
$$
	M = D(F) \cup_{E} D(H).
$$
\end{prop}

\begin{proof}

First assume $\dim C^* = 3$. It follows from Lemma~\ref{L:FPH_SC} that $C$ is fixed by the $\Ss^1$-action, so $C^*$ is the projection of a second fixed point component of maximal dimension $3$. By the Double Soul Theorem we see that $M^*$ is isometric to $F^* \times I$ for some interval $I$ and that $M$ is diffeomorphic to an $\sphere^2$ bundle over $F$. Since $M^*$ is simply connected, $F$ must be simply connected . Furthermore, since $F$ is totally geodesic, $F$ is nonnegatively curved and it follows from Hamilton's work \cite{Ha1986} that $F$ is diffeomorphic to $\sphere^3$. Since a sphere bundle over $\sphere^3$ must be trivial it follows that $M$ is diffeomorphic to 
$\sphere^3 \times \sphere^2$.

Let us now assume $\dim C^* \leq 2$. Since all singularities apart from $F^*$ are contained in $C^*$ it is clear that $F$ is the only fixed point component of dimension $3$.  Note first that the distance functions $d_{C^*}$ and $d_C$ are regular on the regular parts of $M^*$ and $M$, respectively. We distinguish several cases.
\\

\noindent \textit{Case 1.} Assume there exists a one-dimensional fixed point component $S$. Let $H := C = \pi^{-1}(C^*)$. We will show that $H$ is a three-dimensional closed submanifold of $M$. To see this first assume that there occur nonisolated singularities in $C^*$ wich do not correspond to fixed points. For dimensional reasons there is some $2$-dimensional singular set, $K^*$,  corresponding to a $3$-dimensional fixed point component $K$ of some $\mathbb{Z}_k \subset \Ss^1$. Since $\dim C^* \leq 2$ it follows easily from  Kleiner's isotropy  lemma that $C^* = K^*$, so $H = K$. Hence we may assume that there occur only isolated singularities in the interior of $C^*$. Let $p^* \in S^*$ and $[v] \in \Sigma_{p^*}C^*$ with $\angle ([v],S^*) = \pi/2$. Choose a shortest geodesic from $C^*$ to $F^*$ starting at $p^*$ with initial direction $[w] \in \Sigma_{p^*}C^*$. We have $v,w \in \sphere^3 \subset \nu_pS$ and $\angle(\mathrm{S^1}v,\mathrm{S^1}w) = \pi/2$. Since $\mathrm{S^1}$ acts without fixed points on $\sphere^3 \subset \nu_pS$ it follows that $\mathrm{S}^1v$ and $\mathrm{S}^1w$ are contained in 2-dimensional orthogonal subspaces $V,W \subset \nu_pS$, respectively. Therefore a neighbourhood of $p$ in $H$ is contained in $\exp_p(T_pS \oplus V)$. Since $H$ is three-dimensional we see that $H$ is a smooth submanifold in a neighbourhood of $S$. By Proposition~\ref{P:restr_C} it follows that $H$ is a smooth submanifold globally.

As the distance function $d_H$ is regular on $M\setminus(F \cup H)$ one can now define a gradient-like vector field on $M\setminus(F \cup H)$ which is radial near $F$ and $H$, and by this a diffeomorphism $M \cong D(F) \cup_{E} D(H)$ (cf. \cite{Gr1993}).
\\

\noindent\textit{Case 2.} Assume now that $F$ is the only fixed point set component and non-isolated exceptional orbits occur. As in case 1, set $H := C$, wich is a fixed point component of some $\mathbb{Z}_k\subset \mathrm{S}^1$ and hence a closed submanifold.  Similarly, we see that $M \cong D(F) \cup_{E} D(H)$.
\\

\noindent\textit{Case 3.} Assume that $F$ is the only fixed point set component and if exceptional orbits occur, then they are isolated.
\\

\textit{Case 3.1.} Suppose that $\partial C^*=\emptyset$. Because of simple-connectivity we have, up to homeomorphism, $C^* = \{\,p^*\,\}$ or $C^* = \sphere^2$, for some point $p^* \in M^*$. In both cases we define $H =C$ and conclude that $M = D(F) \cup_{E} D(H)$ as above.
\\

\textit{Case 3.2.} Suppse that $\partial C^* \neq \emptyset$. In this case $C^*$ is homeomorphic to a $2$-disc or to a closed interval. We analyze each case separately.
\\

\textit{Case 3.2.1.} $C^*$ is homeomorphic to a $2$-disc.\\
Let 
$$
C^*_1 := \{\, x^*\in C^* : x^* \text{ is at  maximal distance from } \partial C^*\},
$$
so that $C^*_1$ is the next step of the soul construction. This yields a closed interval or a point. If there is an isolated singularity $p^*$ in $C^*$ we have either $C^*_1 = \{p^*\}$ or $p^*$ is a boundary point of $C^*_1$.
\\

\textit{Case 3.2.1.1:} Assume that $C^*_1$ is a point  $p^*$. In this case the distance function $d_{p^*}$ is regular on $M^*\setminus (F^* \cup \{\,p^*\,\})$ so we may define $H := \pi^{-1}(p^*)=\G p$.

\textit{Case 3.2.1.2:} Assume that $C^*_1$ is a closed interval with corresponding isotropy groups
$$
\mathbb{Z}_k \dotsb 1 \dotsb \mathbb{Z}_l.
$$
We then get $\{k,l\} \subset \{1,2\}$. If we have $k = l = 2$, we set $H = \pi^{-1}(C^*_1)$ and are done. If the isotropy at one of the endpoints of $C_1^*$ corresponds to $1$, the distance function from the second endpoint, say $p^*$, is regular outside $(F^* \cup \{\,p^*\,\})$. Thus we define $H$ as $\pi^{-1}(\{\,p^*\,\})=\G p$ and are again done.
\\

\textit{Case 3.2.2.} $C^*$ is homeomorphic to a closed interval. Here we proceed as in case \textit{3.2.1.2.}
\\

Because of the restrictions to the possible dimensions of singular components we now have covered all possible singularity configurations. 
Next we show that the case $\dim H = 2$ cannot occur and that $H$ is orientable, which is nontrivial in case $\dim H = 3$.
\\

Let $\dim H = 2$. Since $H$ has codimension $3$ in $M$ we may excise $H$ without changing the fundamental group of $M$ (cf. Lemma~\ref{L:FPH_SC}). Thus we get 
$$
	0 = \pi_1(M\setminus H) = \pi_1((D(F) \cup_{E} D(H))\setminus H) = \pi_1(F).
$$
Since $F$ is nonnegatively curved,  $F \cong \sphere^3$ by Hamilton's work \cite{Ha1986}. Recall that $E = \partial D(F).$ The exact sequence of homotopy groups for the fiber bundle $\sphere^1 \rightarrow E \rightarrow F$ then gives $\pi_1(E) = \mathbb{Z}$. On the other hand $E$ is an $\sphere^2$-bundle over $H$, so from the fiber bundle $\sphere^2 \rightarrow E \rightarrow H$ we see that $\pi_1(H) = \mathbb{Z}$. Since $H$ is compact and 2-dimensional this gives a contradiction as there exists no compact $2$-manifold with fundamental group $\mathbb{Z}$. Alternatively one can easily see from the construction of $H$ that in the case $\dim H =2$ we get $H \cong \mathbb{K}^2$, the Klein bottle, which has different fundamental group.
\\

It remains to show that $H$ is orientable for $\dim H = 3$. Observe that if $H$ is fixed by some finite (nontrivial) subgroup of $\Ss^1$, then $H$ is orientable. Hence we may assume for simplicity that  $\Ss^1$ acts effectively on $H$. We proceed by contradiction and  assume that $H$ is non-orientable. Let $\tilde{H}$ be the total space of the orientable double cover $\kappa : \tilde{H} \rightarrow H$. Recall that $\tilde{H}$ is connected. Consider the effective $\Ss^1$-action on the pull-back bundle $\kappa^* (\nu H)$ covering the given action on $\nu H$ (cf. \cite{Br1972}). Let $\tilde{q} \in \tilde{H} \subset \kappa^*(\nu H)$ with $\kappa(\tilde{q}) = q$ and $\Ss^1_q = \{e\}$. Since $\Ss^1q$ defines an orientation preserving loop in $H$ we see that $\kappa _{|\Ss^1\tilde{q}}$ is injective. So, if $\Gamma$ denotes the group of deck transformations of $\kappa^* (\nu H)$, it follows that $\Gamma \cap \Ss^1$ equals the identity. From \cite[Theorem 9.1, Ch. 1]{Br1972} we see that the projection
$$
	\tilde{\kappa} : \kappa^*( \nu H)  \rightarrow \nu H
$$
 is an equivariant double covering map. This yields a double cover 
 $$
 \kappa^*(\nu^1 H )\times_{\Ss^1} D^2 \rightarrow \nu^1 H \times_{\Ss^1} D^2 \cong D(F).
 $$
 So one can construct a double covering map 
 $$
 (\kappa^*(\nu^{\leq1} H)) \cup (\kappa^*(\nu^1 H) \times_{\Ss^1} D^2) \rightarrow (\nu^{\leq1} H) \cup (\nu^1 H \times_{\Ss^1} D^2) \cong M^5.
 $$
 Since $M$ is simply connected the total space cannot be connected, wich is a contradiction. This concludes the proof of the proposition.
 \end{proof}

Now we are able to compute $H_2(M^5,\mathbb{Z})$, thus concluding the proof of the Main Theorem.


\begin{prop}
Let $M$ be a closed, simply connected nonnegatively curved $5$-manifold with a fixed point homogenous $\Ss^1$-action. 
Then $H_2(M,\mathbb{Z}) = 0$ or $\mathbb{Z}$.
\end{prop}

\begin{proof}
 Recall that $E := \partial D(F)$. In the case when $\dim H = 1$, we see that $F \cong \sphere^3$ and $H_1(E,\mathbb{Z}) = \pi_1(E) = \mathbb{Z}$ as above. Since $H$ is compact we have $H \cong \sphere^1$. From the Mayer Vietoris sequence we get the exact sequence
$$
\dots \rightarrow 0 \rightarrow H_2(M,\mathbb{Z}) \rightarrow \mathbb{Z} \rightarrow \mathbb{Z} \rightarrow 0
$$
and conclude that $H_2(M,\mathbb{Z}) = 0$.

Suppose now that $\dim H = 3$. First note that since $H$ and $M$ are orientable it is possible to define a second $\Ss^1$-action on the normal bundle $\nu H$ of $H$ by rotating the fibers. We will denote this second circle acting on $\nu H$ by $\Ss^1_2$. This action commutes with the initial fixed point homogeneous $\Ss^1$-action given on $M$, which we shall denote by $\Ss^1_1$. In this way  we get an action of $\mathrm{T}^2=\Ss^1_1\oplus\Ss^1_2$  on  $\nu H$, and thus on $D(H)$ via the exponential map. We can extend this action to $D(F)$ and $M$ in the following way. The boundary of $D(H)$ is invariant under the $\mathrm{T}^2$-action and diffeomorphic to $E = \partial D(F)$, via the gluing map of the decomposition $M \cong D(F) \cup_E D(H)$, so we can pull back the action to $E$. Note that we have $D(F) \cong E \times_{\Ss^1_1} D^2$ and $D(H) \cong E \times_{\Ss^1_2} D^2$, where $\Ss^1_i$ acts on the unit disc $D^2$ from the right via rotation.

Now we can extend the action defined on $E$ to $E \times_{\Ss^1_1} D^2$ by 
$$
(\theta_1,\theta_2) \cdot [p,v]_1 := [(\theta_1,\theta_2) \cdot p,v]_1
$$ 
for $p \in E, v \in D^2$ and $(\theta_1,\theta_2) \in \mathrm{T^2} = \Ss^1_1 \oplus \Ss^1_2$. Thus we get a $\mathrm{T}^2$-action on $(E \times_{\Ss^1_1} D^2) \cup (E \times_{\Ss^1_2} D^2) \cong M$. By the Hurewicz theorem we have $H_2(M,\mathbb{Z}) \cong \pi_2(M)$, so we need to calculate the latter group. 

Let $\Ss^1_1$ and $\Ss^1_2$ act freely from the right on $\sphere^3$ and consider the induced free action of $ \mathrm{T}^2=\Ss^1_1 \oplus \Ss^1_2 $ on $\sphere^3 \times \sphere^3$. In order to calculate $\pi_2(M)$ it is convenient to consider 
$$\hat{M} \cong (E \times \sphere^3 \times \sphere^3) \times_{\mathrm{S}^1_1} D^2 \cup (E \times \sphere^3 \times \sphere^3) \times_{\mathrm{S}^1_2} D^2$$
together with the $\mathrm{T}^2$-action 
$$
t \cdot [(p,x,y),v]_i := [(t \cdot p,(x,y) \cdot t^{-1}),v]_i,
$$ for $i \in \{1,2\}, t \in \mathrm{T}^2, p \in E, (x,y) \in \sphere^3 \times \sphere^3$ and $v \in D^2$. Then we have a fibration $\sphere^3 \times \sphere^3 \rightarrow \hat{M} \rightarrow M$ coming from the projection $[(p,x,y),v]_i \mapsto [p,v]_i$ and we see that $\pi_2(M) \cong \pi_2(\hat{M})$, which we will calculate as follows.

Observe first that $\mathrm{T}^2$ acts freely on $E \times \sphere^3 \times \sphere^3$ and we have a fibration $N^3 \rightarrow \hat{M} \rightarrow E \times_{\mathrm{T}^2} \sphere^3 \times \sphere^3$ coming from the map $[(p,x,y),v]_i \mapsto [p,(x,y)]$. Now we show that the fiber $N^3$ is diffeomorphic to a 3-sphere.

Let $q = [p_0,(x_0,y_0)] \in E \times_{\mathrm{T}^2} \sphere^3 \times \sphere^3$. The fiber over $q$ is given by 
$$
	\{\, [(\theta_1,\theta_2)\cdot p_0,x_0\cdot\theta_1^{-1},y_0\cdot\theta_2^{-1},v]_i\  :\ (\theta_1,\theta_2) \in T^2, v \in D^2, i \in \{1,2\} \, \},
$$
which is clearly invariant under the $\mathrm{T}^2$ action on $\hat{M}$. A point in $[(p,x,y),v]_i \in \hat{M}$ has nontrivial isotropy if and only if $v = 0$ and  for points $\overline{p}_i = [(p_0,x_0,y_0),0]_i$, $i=1,2$, we have $\mathrm{T}^2_{\overline{p}_1} = \Ss^1_1 \oplus 0$ and $\mathrm{T}^2_{\overline{p}_2} = 0 \oplus \Ss^1_2$. It follows that $N^3/\mathrm{S}^1_1$ is a smooth $2$-manifold with a circle as boundary admitting a smooth $\Ss^1_2$ action fixing a unique point. Hence $N^3/\Ss^1_1$ is a 2-disc whose boundary and interior points correspond to fixed points and principal orbits respectively. It follows from the the classification of smooth $3$-manifolds with a smooth circle action  \cite{Ra,ORa} that $N^3$ is diffeomorphic to $\sphere^3$ (cf. also \cite{GroveSearle1997}).

Therefore we have $\pi _1(N^3) = \pi_2(N^3) = 0$ and from the exact sequence for the fibration $N^3 \rightarrow \hat{M} \rightarrow E \times_{\mathrm{T}^2} \sphere^n \times \sphere^n$ we get an isomorphism
$$
	\pi_2(\hat{M}) \cong \pi_2(E \times_{\mathrm{T}^2} \sphere^3 \times \sphere^3).
$$
So we have to show that the last group is either isomorphic to $\mathbb{Z}$ or $0$.

Since $\hat{M}$ is simply connected $\pi_1(E \times_{\mathrm{T}^2} \sphere^3 \times \sphere^3) = 0$ and we get the following three exact sequences:
\begin{align}
\label{E:les_1}
0 \rightarrow \pi_2(E \times \sphere^3 \times \sphere^3) \rightarrow \pi_2(E \times_{\mathrm{T}^2} \sphere^3 \times \sphere^3) \rightarrow \pi_1(\mathrm{T}^2) \xrightarrow{j} \pi_1(E \times \sphere^3 \times \sphere^3) \rightarrow 0,
\end{align}
\begin{align}
\label{E:les_2}
0 \rightarrow \pi_2(E) \rightarrow \pi_2(F) \rightarrow \pi_1(\Ss^1_1) \xrightarrow{i_1} \pi_1(E) \xrightarrow{p_1} \pi_1(F) \rightarrow 0,
\end{align}
\begin{align}
\label{E:les_3}
0 \rightarrow \pi_2(E) \rightarrow \pi_2(H) \rightarrow \pi_1(\Ss^1_2) \xrightarrow{i_2} \pi_1(E) \xrightarrow{p_2} \pi_1(H) \rightarrow 0.\nonumber
\end{align}

Let 
$$
	k = i_1 \oplus i_2 : \pi_1(\Ss^1_1) \oplus \pi_1(\Ss^1_2) \longrightarrow \pi_1(E).
$$ 
Then $k$ is surjective, since $j$ is surjective and $\ker k \cong \ker j$. From the surjection $\pi_1(\Ss^1_2) \xrightarrow{i_{2}} \pi_1(E)/\operatorname{im}(i_1) \cong \pi_1(F)$ it follows that $\pi_1(F)$ is cyclic, so we distinguish two cases. 
\\

\noindent\textit{Case 1:} $\pi_1(F) \cong \mathbb{Z}_n$ for some $n$. Since $F$ has nonnegative curvature, by the work of Hamilton we know that $F$ is diffeomorphic to a spaceform $S^3/\mathbb{Z}_n$. Thus $\pi_2(F) = 0$. Concluding $\pi_2(E) = 0$ yields an isomorphism 
$$
\pi_2(E \times_{\mathrm{T}^2} \sphere^3 \times \sphere^3) \cong \operatorname{ker}(j).
$$ 
Furthermore $i_1$ is injective, yielding that $\ker k \cong \ker j$ is either $0$ or isomorphic to $\Int$.
\\

\noindent\textit{Case 2:} $\pi_1(F) \cong \mathbb{Z}$. Consider the universal cover $\tilde{F}$ of $F$. Since $F$ has nonnegative curvature it follows from the splitting theorem that $\tilde{F}$ is diffeomorphic to $\mathbb{R} \times \sphere^2$ so we see that $\pi_2(F) \cong \mathbb{Z}$. From (\ref{E:les_2}) we get
$$
	0 \rightarrow \pi_2(E) \rightarrow \mathbb{Z} \xrightarrow{l} \mathbb{Z} \xrightarrow{i_1} \pi_1(E) \xrightarrow{p_1} \mathbb{Z} \rightarrow 0.
$$

\textit{2.1.} Let $l \equiv 0$. Hence we have $\pi_2(E) \cong \mathbb{Z}$ and $\pi_1(E)/i_1(\mathbb{Z}) \cong \mathbb{Z}$, where $i_1$ is injective. From the surjectivity of $k$ we then see that $\pi_1(E) \cong \mathbb{Z}^2$.
Thus (\ref{E:les_1}) gives
$$
0 \rightarrow \mathbb{Z} \rightarrow \pi_2(E \times_{\mathrm{T}^2} \sphere^3 \times \sphere^3) \rightarrow \mathbb{Z}^2 \xrightarrow{j} \mathbb{Z}^2 \rightarrow 0.
$$
By exactness, $\pi_2(E \times_{\mathrm{T}^2} \sphere^3 \times \sphere^3) \cong \mathbb{Z}$ since $j$ must be surjective and hence an isomorphism.

\textit{2.2.}Assume $l$ is injective, so we have $\pi_2(E) = 0$ and $\pi_2(E \times_{\mathrm{T}^2} \sphere^3 \times \sphere^3) \cong \ker{j}$.

\textit{2.2.1.} If $l$ is surjective we have $i_1 \equiv 0$ and it follows that $\pi_1(E) \cong \mathbb{Z}$ and $\ker j \cong \ker k \cong \mathbb{Z}$.

\textit{2.2.2.} If $l$ is not surjective we get $\ker i_1 = a\mathbb{Z} \subset \mathbb{Z}$ for some $a \geq 2$. Hence $\pi_1(E)$ has torsion and we conclude that $\pi_1(E) \cong \mathbb{Z} \oplus \mathbb{Z}_a$. Again we see that $\ker j \cong \ker k \cong \ker i_1 \cong \mathbb{Z}$.

\end{proof}

Having calculated the second integral homology group of $M^5$ the Main Theorem follows from the Barden-Smale classification of closed, simply-connected $5$-manifolds \cite{Ba1965,Sm1962}, as stated in the introduction.

\bibliographystyle{amsplain}

\begin{thebibliography}{10}

\bibitem{Ba1965}
D.~Barden, \emph{Simply connected five-manifolds}, Annals of Math. \textbf{82}
  (1965), no.~3, 365--385.

\bibitem{BB}L. Berard-Bergery, \emph{Les vari\'et\'es Riemanniennes homog\`enes simplement connexes de dimension impaire  
\'a courboure strictement positive}, J. Math. Pures Appl.  (9)  55  (1976), no. 1, 47--67.

\bibitem{Br1972}
G.~E. Bredon, \emph{Introduction to compact transformation groups}, Pure and
  Applied Mathematics, vol.~46, Academic Press, New York-London, 1972.

\bibitem{BBI}
D.~Burago, Yuri B., and S.~Ivanov, \emph{A course in metric geometry}, Graduate
  Studies in Mathematics, vol.~33, American Mathematical Society, Providence,
  RI, 2001.

\bibitem{ChGr1972}J.~Cheeger and D.~Gromoll, \emph{On the structure of complete manifolds of nonnegative curvature}, Ann. of Math. (2) 96 (1972), 413--443. 

\bibitem{Ga2}F.~Galaz-Garcia,  \emph{Nonnegatively curved fixed point homogeneous manifolds in low dimensions}, Geom. Dedicata, to appear. arXiv:0911.1254v1 [math.DG]


\bibitem{GaGaSe2010}
F.~Galaz-Garcia and C.~Searle, \emph{Nonnegatively curved $5$-manifolds with almost  maximal symmetry rank}, Preprint (2011) (see arXiv:0906.3870v1 [math.DG]).

\bibitem{Gr1993} 
K.~Grove, \emph{Critical point theory for distance functions}. Differential geometry: Riemannian geometry (Los Angeles, CA, 1990), 357--385, Proc. Sympos. Pure Math., 54, Part 3, Amer. Math. Soc., Providence, RI, 1993. 

\bibitem{Gr2002}
K.~Grove, \emph{Geometry of, and via, symmetries}, Conformal, Riemannian and Lagrangian geometry (Knoxville, TN, 2000), 31--53, Univ. Lecture Ser., 27, Amer. Math. Soc., Providence, RI, 2002.

\bibitem{GroveSearle1994} K.~Grove and C.~Searle, \emph{Positively curved manifolds with maximal symmetry-rank}, J. Pure Appl. Algebra  91  (1994), no. 1-3, 137--142.

\bibitem{GroveSearle1997}
 K.~Grove and C.~Searle, \emph{Differential topological restrictions curvature
  and symmetry}, J. Differential Geom. \textbf{47} (1997), no.~3, 530--559.

\bibitem{GWZ}
K.~Grove, B.~Wilking, and W.~Ziller, \emph{Positively curved cohomogeneity one
  manifolds and 3-{S}asakian geometry}, J. Differential Geom. \textbf{78}
  (2008), no.~1, 33--111. 

\bibitem{GZ1} K. Grove and W. Ziller, \emph{Curvature and symmetry of Milnor spheres}, Ann. of Math. (2),  \textbf{152}  (2000),  no. 1, 331--367. 


\bibitem{Ha1986}
R.~S. Hamilton, \emph{Four-manifolds with positive curvature operator}, J.
  Differential Geom. \textbf{24} (1986), no.~2, 153--179.

\bibitem{Ho2007}
C.~Hoelscher, \emph{Classification of cohomogeneity one manifolds in low
  dimensions}, Pacific J. Math. 246 (2010), no. 1, 129--185.

\bibitem{Ko2006}
J.~Koll\'ar, \emph{Circle actions on simply connected $5$-manifolds}, Topology
  \textbf{45} (2006), no.~3, 643--671.


\bibitem{ORa}
P.~Orlik and F.~Raymond, \emph{Actions of {${\rm SO}(2)$} on 3-manifolds},
  Proc. {C}onf. on {T}ransformation {G}roups ({N}ew {O}rleans, {L}a., 1967),
  Springer, New York, 1968, pp.~297--318.

\bibitem{Pe2002}
G.~Perelman, \emph{The entropy formula for the {R}icci flow and its geometric  applications}, Preprint (2002), arXiv:math/0211159v1 [math.DG].

\bibitem{Pe2003}
G.~Perelman, \emph{Ricci flow with surgery on three-manifolds}, Preprint (2003),  arXiv:math/0303109v1 [math.DG].

\bibitem{Ra}
F.~Raymond, \emph{Classification of the actions of the circle on
  {$3$}-manifolds}, Trans. Amer. Math. Soc. \textbf{131} (1968), 51--78.


\bibitem{S} C. Searle, \emph{Cohomogeneity and Positive Curvature in Low Dimensions}, Math. Z.,  \textbf{214}  (1993),  no. 3, 491--498; {\it Corrigendum}, Math. Z., \textbf{214} (1993), no. 3, 491--498.

\bibitem{SY} C. Searle and D. Yang, \emph{On the topology of nonnegatively curved simply-connected 4-manifolds with continuous symmetry}, Duke Math. J, vol. 74, no. 2, 547 -- 556 (1994).


\bibitem{Sm1962}
S. Smale, \emph{On the structure of {$5$}-manifolds}, Ann. of Math. (2)
  \textbf{75} (1962), 38--46.

\bibitem{Ve2002} L. Verdiani, \emph{Cohomogeneity one Riemannian manifolds of even dimension with strictly positive sectional curvature, I},  Math. Z, \textbf{241}  (2002), 329--229. 

\bibitem{Ve2004}
L.~Verdiani, \emph{Cohomogeneity one manifolds of even dimension with strictly
  positive sectional curvature}, J. Differential Geom. \textbf{68} (2004),
  no.~1, 31--72.


\bibitem{Wa}N.~Wallach, \emph{Compact homogeneous Riemannian manifolds with strictly positive curvature}, Ann. Math. 96 (1972), 277--295.

\bibitem{Wi_sym}
B.~Wilking, \emph{Positively curved manifolds with symmetry}, Ann. of Math. (2)
  \textbf{163} (2006), no.~2, 607--668.
  
\bibitem{Wkg2007}
B.~Wilking, \emph{Nonnegatively and positively curved manifolds}, Surveys in
  differential geometry. {V}ol. {XI}, Surv. Differ. Geom., vol.~11, Int. Press,
  Somerville, MA, 2007, pp.~25--62.

\end{thebibliography}


\end{document}